\documentclass[12pt,a4paper]{amsart}
\usepackage{amssymb}
\usepackage{graphicx}

\newcommand{\N}{\mathbb{N}}
\newcommand{\Z}{\mathbb{Z}}
\newcommand{\Q}{\mathbb{Q}}
\newcommand{\R}{\mathbb{R}}
\newcommand{\C}{\mathbb{C}}
\newcommand{\F}{\mathbb{F}}

\newcommand{\coef}[1][{Q,-k}]{C({#1})}
\newcommand{\coeff}[1][{Q}]{C'({#1})}
\newcommand{\diag}{\operatorname{diag}}
\newcommand{\legendre}[2]{\left(\frac{#1}{#2}\right)}
\newcommand{\ba}{\backslash}
\newcommand{\op}{\operatorname}

\newcommand{\M}{\mathrm{M}}
\newcommand{\GL}{\mathrm{GL}}
\newcommand{\SL}{\mathrm{SL}}

\newcommand{\U}{\mathrm{U}}
\newcommand{\SO}{\mathrm{SO}}
\newcommand{\SU}{\mathrm{SU}}

\newcommand{\PSO}{\mathrm{PSO}}

\theoremstyle{plain}
\newtheorem{theorem}{Theorem}[section]
\newtheorem{corollary}[theorem]{Corollary}
\newtheorem{lemma}[theorem]{Lemma}

\newtheorem{conjecture}[theorem]{Conjecture}
\theoremstyle{definition}
\newtheorem{definition}[theorem]{Definition}
\theoremstyle{remark}
\newtheorem{remark}[theorem]{Remark}
\newtheorem{example}[theorem]{Example}

\numberwithin{equation}{section}
\numberwithin{theorem}{section}

\title[A numerical study on exceptional eigenvalues]
{A numerical study on exceptional eigenvalues of certain congruence subgroups of $\SO(n,1)$ and $\SU(n,1)$}
\author{Emilio A. Lauret}
\address{FaMAF--CIEM \\ Universidad Nacional de C\'ordoba.  X5000HUA --- C\'ordoba, Argentina}
\email{elauret@famaf.unc.edu.ar}
\thanks{Supported by CONICET and Secyt-UNC}
\date{December 15, 2012}

\begin{document}
\begin{abstract}
In a previous work we apply lattice point theorems on hyperbolic spaces obtaining asymptotic formulas for the number of integral representations of negative integers by quadratic and hermitian forms of signature $(n,1)$ lying in Euclidean balls of increasing radius.
This formula involves an error term that depends on the first nonzero eigenvalue of the Laplace-Beltrami operator on the corresponding congruence hyperbolic manifolds.
The aim of this paper is to compare the error term obtained by experimental computations, with the error term mentioned above, for several choices of quadratic and hermitian forms.

Our numerical results give evidences of the existence of exceptional eigenvalues for some arithmetic subgroups of $\SU(3,1)$, $\SU(4,1)$ and $\SU(5,1)$, thus they contradicts the generalized Selberg (and Ramanujan) conjecture in these cases.
Furthermore, for several arithmetic subgroups of $\SO(4,1)$, $\SO(6,1)$, $\SO(8,1)$ and $\SU(2,1)$, there are evidences of a lower bound on the first nonzero eigenvalue better than the already known lower bound for congruences subgroups.
\end{abstract}

\maketitle

\section{Introduction}

Let $\mathrm H_\F^n$ denote the $\F$-hyperbolic space (for $\F=\R$, $\C$ or $\mathbb H$) of dimension $n$.
Let $x,y\in\mathrm H_\F^n$ and let $\Gamma$ be a discrete subgroup of the isometry group of $\mathrm H_\F^n$ such that $\Gamma\ba \mathrm H_\F^n$ has finite volume.
Hyperbolic lattice point theorems give asymptotic formulas (as $r\to+\infty$) for
\[
N_{x,y}(r) := \#\{\gamma\in\Gamma: d(x,\gamma\cdot y)\leq r\}.
\]
This problem has been investigated by several authors (see for instance \cite{Patterson}, \cite{Lax-Phillips}, \cite{Levitan}, \cite{PhillipsRudnick}, \cite{Bruggeman-Miatello-Wallach}, \cite{Gorodnik-Nevo}), most of them concentrate on real hyperbolic spaces.
The main theme in these works links the asymptotic behavior of $N_{x,y}(r)$ with spectral data of the Laplace-Beltrami operator on $\Gamma\ba \mathrm H_\F^n$.

For example, if $\Gamma\ba \mathrm H_\F^n$ is not compact, Bruggeman, Miatello and Wallach~\cite{Bruggeman-Miatello-Wallach} showed that
\begin{equation}\label{eq1:LPTofBMW}
  N_{x,y}(r) =
    \frac{2^{1-n}\op{vol}(S^{rn-1})}{2\rho\op{vol}(\Gamma\ba\mathrm H_\R^n)} \, e^{2\rho\, r}
    + \sum_{j=1}^N B_j\; \varphi_j(x)\varphi_j(y) \, e^{\left(\rho+\nu_j\right)r}
    +O\left(e^{(2\rho\frac{n}{n+1}+\varepsilon)r}\right)
\end{equation}
for all $\varepsilon>0$, where
\begin{eqnarray*}
r=r_\F=\dim_\R(\F)&=& 1,2,4, \\
\rho=\rho_{n,\F}=\tfrac{n+1}{2}\, r-1 &=& \tfrac{n-1}2, n,2n+1
\end{eqnarray*}
for $\F=\R,\C,\mathbb H$ respectively, $\op{vol}(S^{m})$ denotes the volume of the $m$-dimensional unit sphere, $B_1,\dots, B_N$ are positive (explicit) coefficients, the numbers
\[
0<\lambda_1\leq\dots\leq\lambda_N< \rho^2
\]
are the \emph{exceptional eigenvalues} of the Laplace-Beltrami operator on $\Gamma\backslash\mathrm H_\F^n$ counted with multiplicities and, for each $1\leq j\leq N$,
\[
\nu_j= \sqrt{\rho^2-\lambda_j}
\]
and $\varphi_j$ is the normalized eigenfunction associated to $\lambda_j$.
We recall that $f(x)=O(g(x))$ if and only if there are $M,C>0$ such that $|f(x)|\leq C\,g(x)$ for all $x>M$; this is also denoted by $f(x) \ll g(x)$.
If $\F=\R$, \eqref{eq1:LPTofBMW} holds already with $\varepsilon=0$ by \cite{Levitan}.
In general, $\lambda_j$ and $\varphi_j$ are not determined for any $j$, hence it is customary to work with the abbreviated formula
\begin{equation}\label{eq1:LPTofBMW_tau}
  N_{x,y}(r) =
    \frac{2^{1-n}\op{vol}(S^{rn-1})}{2\rho\op{vol}(\Gamma\ba\mathrm H_\R^n)} \, e^{2\rho\, r}
    +O\left(e^{\tau r}\right),
\end{equation}
where
\begin{equation}\label{eq1:tau}
\tau=\tau(\Gamma):=\begin{cases}
\nu_1+\rho
  &\quad\text{if }\nu_1+\rho > 2\rho \frac{n}{n+1},\\
2\rho\frac{n}{n+1}+\varepsilon
  &\quad \text{otherwise}.
\end{cases}
\end{equation}

We can see that the exceptional eigenvalues play an important role in the distribution of the lattice points.
The existence or not of exceptional eigenvalues is an open problem.
By definition, $\lambda_1(\Gamma\ba\mathrm H_\F^n)\geq\rho_{n,\F}^2$ if and only if $\Gamma\ba\mathrm H_\F^n$ has not exceptional spectrum.
The \emph{Selberg eigenvalue conjecture} says that
\[
\lambda_1(\Gamma\ba \mathrm H_\R^2)\geq \rho_{2,\R}^2 = 1/4
\]
for all congruence subgroups $\Gamma$ in $\SL(2,\Z)$.
This is part of the \emph{Ramanujan-Petersson conjecture} at infinity.
Selberg~\cite{Selberg} established the bound $\lambda_1\geq 3/16$ by using the Kloosterman-Selberg zeta function and Weil estimates for Kloosterman sums.
The best lower bound available at the moment is $\lambda_1\geq 975/4096\approx 0.238\dots$ proved by Kim and Sarnak~\cite{Kim-Sarnak}.
Furthermore, the conjecture has been verified in some explicit cases: Roelcke~\cite{Roelcke} for $\Gamma=\SL(2,\Z)$ (see also \cite{Deshouillers-Iwaniec}), Huxley~\cite{Huxley} for $\Gamma_0(N)$ for all $N<19$, and more recently, Booker and Str\"ombergsson~\cite{Booker-Stroembergsson} for $\Gamma_0(N)$ for all $N<857$ squarefree.

It is also expected that this conjecture holds for $n=3$; in this case the conjecture says that $\lambda_1(\Gamma\ba \mathrm H_\R^3)\geq \rho_{3,\R}^2=1$ and Sarnak~\cite{Sarnak} proved that $\lambda_1(\Gamma\ba \mathrm H_\R^3)\geq 3/4$ for all congruence subgroups of $\SL(2,\mathcal O)$, where $\mathcal O$ is the ring of integers of an imaginary quadratic number field.
Some particular cases are verified in \cite{EGMsoviet}.

It is known that the generalization of Selberg's conjecture in higher dimensions is false.
When $\F=\R$ and $n$ is arbitrary, for $Q$ a quadratic form with rational coefficients of signature $(n,1)$, we say that a subgroup of
\begin{equation}\label{eq1:Gamma_Q^0}
\Gamma_Q^0:=\GL(n+1,\Z)\cap \SO^0(Q)
\end{equation}
(see Subsection~\ref{subsec:hyp_spaces} for arbitrary $\F$) is a congruence subgroup if it contains $\Gamma_Q^0(N):=\{g\in\Gamma_Q^0: g\equiv I_{n+1}\pmod N\}$ for some $N\in\N$ (see Definition~\ref{def2:cong-subg}  for arbitrary $\F$).
One can find in \cite[\S 6]{Cogdell-Li-Piatetski-Shapiro-Sarnak} an explicit construction via theta-lifting from $\SL(2)$ of exceptional spectrum of $\Gamma_Q^0(N)\ba\mathrm H_\R^n$ for large $N$.
However, Selberg and Sarnak's lower bound were generalized to
\begin{equation}\label{eq1:upper_tau_R}
\lambda_1(\Gamma\ba \mathrm H_\R^n)\geq \frac{2n-3}4
\end{equation}
for all $n\geq 3$ and for any congruence subgroup $\Gamma$.
This was proved independently by Elstrodt, Grunewald and Mennicke~\cite{EGMKloosterman} and Cogdell, Li, Piatetski-Shapiro and Sarnak~\cite{Cogdell-Li-Piatetski-Shapiro-Sarnak}.

In the unitary case the situation becomes even darker.
Jian-Shu Li~\cite{Li} (see Theorem~\ref{thm2:lambda_1_C}) proved that the first eigenvalue $\lambda_1$ for the Laplace-Beltrami operator on the subspace of nondegenerate forms (nontrivial Fourier coefficients) in $L^2(\Gamma\ba \mathrm H_\C^n)$ satisfies $\lambda_1\geq 2n-1$, where $\Gamma$ is any congruence subgroup of $\SU(n,1)$.

The aim of this paper is to check \eqref{eq1:LPTofBMW} numerically by assuming that there are no exceptional eigenvalues in $\Gamma_Q^0\ba\mathrm H_\F^n$ for several choices of $Q$.
Our computations seem to contradict this assumption ($\lambda_1\geq \rho^2$) for some hermitian forms $Q$ of signature $(n,1)$ with $n=3,4,5$.
In other words, we have numerical evidences to affirm that $\Gamma_Q^0\ba\mathrm H_\F^n$ has exceptional spectrum.

We also consider quadratic forms of signature $(n,1)$ for $n=2,4,6,8$ and hermitian forms of signature $(2,1)$.
In these cases, our numerical results do not seem to contradict the assumption $\lambda_1\geq \rho^2$.
Hence, our calculations give evidences of the weaker condition
\begin{equation}\label{eq1:weaker_cond}
\lambda_1(\Gamma_Q^0\ba \mathrm H_\F^n) > \frac{4n}{(n+1)^2}\; \rho_{n,\F}^2
\qquad \qquad
\left(\text{if and only if}\quad \rho+\nu_1\leq 2\rho\,\frac{n}{n+1}\right),
\end{equation}
since, by \eqref{eq1:tau}, this is the minimum lower bound that ensure the smallest error term in \eqref{eq1:LPTofBMW_tau}.
If \eqref{eq1:weaker_cond} holds (e.g.\ there are no exceptional eigenvalues), then we have that \eqref{eq1:LPTofBMW_tau} holds with
\begin{equation}\label{eq1:tau_sin_autov_excep}
\tau=2\rho\;\frac{n}{n+1}+\varepsilon=
\begin{cases}
  n(n-1)/(n+1)  &\quad\text{if $\F=\R$},\\
  2n^2/(n+1)+\varepsilon &\quad\text{if $\F=\C$,}\\
  (4n+2)n/(n+1)+\varepsilon& \quad\text{if $\F=\mathbb H$,}
\end{cases}
\end{equation}
for every $\varepsilon>0$ (recall that $\varepsilon=0$ when $\F=\R$).
To check \eqref{eq1:LPTofBMW} in a numerical way, we use a recent arithmetic application.

Hyperbolic lattice point theorems have been applied in different contexts, mostly on arithmetic problems (see \cite{EGMarith_applic}).
For instance, Ratcliffe and Tschantz~\cite{Ratcliffe-Tschantz}, by applying the results of \cite{Lax-Phillips} and \cite{Levitan}, obtained the asymptotic formula, as $r\to+\infty$, for the number of representations of a negative integer by the Lorentzian quadratic form that are contained in the ball of radius $r$ centered at the origin in $\R^{n+1}$.
In \cite{Lauret12}, we generalized \cite{Ratcliffe-Tschantz} to more general integral quadratic forms of signature $(n,1)$ and also to hermitian forms.
We next introduce some notation to state the main formula of \cite{Lauret12}.

For $\F=\R$ or $\C$, let $\mathcal O$ denote a maximal order in $\F$, that is, $\mathcal O=\Z$ in the real case and the ring of integers of an imaginary quadratic extension of the rational numbers if $\F=\C$.
We consider, for $n\geq2$, an integral $\F$-hermitian matrix ($B\in \M(m,\mathcal O)$ such that $B=B^*$)
\begin{equation}\label{eq1:Q}
Q\;=\;
\begin{pmatrix}
A&\\&-a
\end{pmatrix},
\end{equation}
with $a\in\N$ and $A\in\M(n,\mathcal O)$ a positive definite integral $\F$-hermitian matrix.
We also denote by $Q$ the induced $\F$-hermitian form $Q[x]:=x^*Qx=A[\hat x]-a\,|x_{n+1}|^2$ ($x\in\mathcal O^{n+1}$ and $\hat x:=(x_1,\dots,x_{n})^t$).
For $t>0$, put
\begin{equation}\label{eq1:N_t}
N_t(Q,-k)=\{x\in\mathcal O^{n+1}:Q[x]=k,\;|x_{n+1}|\leq t\}.
\end{equation}
In \cite{Lauret12}, by applying the formula \eqref{eq1:LPTofBMW} from \cite{Bruggeman-Miatello-Wallach} to the discrete subgroup $\Gamma_Q^0$ given by \eqref{eq1:Gamma_Q^0}, we establish the following formula.

\begin{theorem}\label{thm1:main_form}
If $-k$ is represented by $Q$, then
\begin{equation}\label{eq1:main-form}
    N_t(Q,-k)=
    C(Q,-k)\;
    t^{2\rho} + O(t^{\tau})
    \qquad \text{as } t\to+\infty,
\end{equation}
where
\begin{equation}\label{eq1:C(Q,-k)}
  C(Q,-k)=    \frac{2^{(r-1)(n+1)} \, a^\rho}{|d_{\mathcal O}|^{\frac{n+1}2} \, |\det Q|^{\frac{r}{2}}}\;
    \frac{\mathrm{vol}(S^{nr-1})}{2\rho}\;
    \frac{\pi^{r/2}}{\Gamma(\frac r2)}\;
    \delta(Q,-k).
\end{equation}
Here, $d_{\mathcal O}$ denotes the discriminant of the quotient field of $\mathcal O$, $\delta(Q,-k)=\prod_p \delta_p(Q,-k)$ is the local density of the representation $Q[x]=-k$ (see~\eqref{eq3:density}) and $\tau$ is as in \eqref{eq1:tau}.
\end{theorem}

Now we can explain our strategy in more detail.
We set
\begin{equation}\label{eq1:Psi}
\Psi(t)=\Psi_t{(Q,-k)}=\left|\frac{N_t(Q,-k)}{t^{2\rho}}-\coef\right|,
\end{equation}
thus $\Psi(t)\ll t^{\tau-2\rho}$ and $\lim_{t\to\infty} \Psi(t)=0$ since $\tau<2\rho$.
With the help of a computer, we quantify $\Psi(t)$ for large $t$, estimating its decay as $\Psi(t) \sim B t^\sigma$ (i.e.\ $\lim_{t\to\infty} \Psi(t)/t^\sigma=B$) for some $\sigma<0$ and $B>0$.
Next, we compare the obtained error term $O(t^{\tau-2\rho})$ of $\Psi(t)$ in \eqref{eq1:main-form} with our estimate $O(t^\sigma)$.
As we mentioned before, we are mostly interested in comparing them by assuming that $\lambda_1(\Gamma_Q^0\ba\mathrm H_\F^n)\geq \rho_{n,\F}$ (i.e.\ there are no exceptional eigenvalues in $\Gamma_Q^n\ba \mathrm H_\F^n$), or more generally by assuming the weaker condition \eqref{eq1:weaker_cond}; in both cases, $\tau$ is given by \eqref{eq1:tau_sin_autov_excep}.

We state the conjectures evidenced from the numerical results.
Let $I_n$ denote the $n\times n$-identity matrix.

\begin{conjecture}\label{conj1:existence}
Let $Q$ be the $\C$-hermitian form $\diag(I_n,-a)$ over $\Z[\sqrt{-3}]$ for $n=3,4,5$, where $a\in\Z$ and $1\leq a\leq 15$.
Let $\Gamma_Q^0\subset\SU(Q,\C)\cong\SU(n,1)$ be the associated discrete subgroup given by \eqref{eq1:Gamma_Q^0} and let $\lambda_1$ be the first nonzero eigenvalue of the Laplace-Beltrami operator on $\Gamma_Q^0\ba\mathrm H_\C^n$.
Then $\Gamma_Q^0\ba\mathrm H_\C^n$ has exceptional spectrum (i.e.\ $\lambda_1<\rho_{n,\C}^2=n^2$).
Moreover,
\[
\lambda_1 < \frac{4n^3}{(n+1)^2}.
\]
\end{conjecture}

\begin{conjecture}\label{conj1:lower_bound}
Let $Q$ be the $\F$-hermitian form $\diag(I_n,-a)$ for $n=4,6,8$ if $\F=\R$ and, over $\Z[\sqrt{-3}]$ for $n=2$ if $\F=\C$, where $a\in\Z$ and $1\leq a\leq 15$.
Let $\Gamma_Q^0$ and $\lambda_1$ as above.
Then
\[
\lambda_1 \geq \frac{4n}{(n+1)^2}\;\rho_{n,\F}^2,
\]
where $\rho_{n,\F}=\frac{n-1}2,n$ for $\F=\R,\C$ respectively.
\end{conjecture}

We have made similar computations for other quadratic forms of signature $(4,1)$ and $(6,1)$.
They are listed at the end of the paper.
Conjecture~\ref{conj1:lower_bound} also holds for these cases.

This paper is organized as follows.
In Section~\ref{sec:exceptional}, since $\Gamma_Q^0$ is trivially a congruence subgroup (see Definition~\ref{def2:cong-subg}), we obtain upper bounds for $\tau$ by applying known lower bounds for $\lambda_1$ for general congruence subgroups.
Sections~\ref{sec:local_densities} is devoted to compute, for arbitrary diagonal nondegenerate $\F$-hermitian forms $Q$ and any $k\in\Z$, the local density $\delta(Q,k)=\prod_p\delta_p(Q,k)$.
In Section~\ref{sec:approx} we compute the main coefficient $C(Q,-k)$ and also we give an approximation for the number $\sigma$.
In Section~\ref{sec:conclusion} we present tables with the numerical results and their respective conclusions.
All the necessary computations have been done in \cite{Sage}.

Let us conclude this introduction by looking at two illustrative examples.

\begin{figure}
\caption{Graphic of $\Psi_t(Q,-1)$ when $\F=\R$, $Q=I_{4,1}$ and $k=1$.}
\label{fig:Q=(I_4,-1)}
\includegraphics[width=0.95\textwidth]{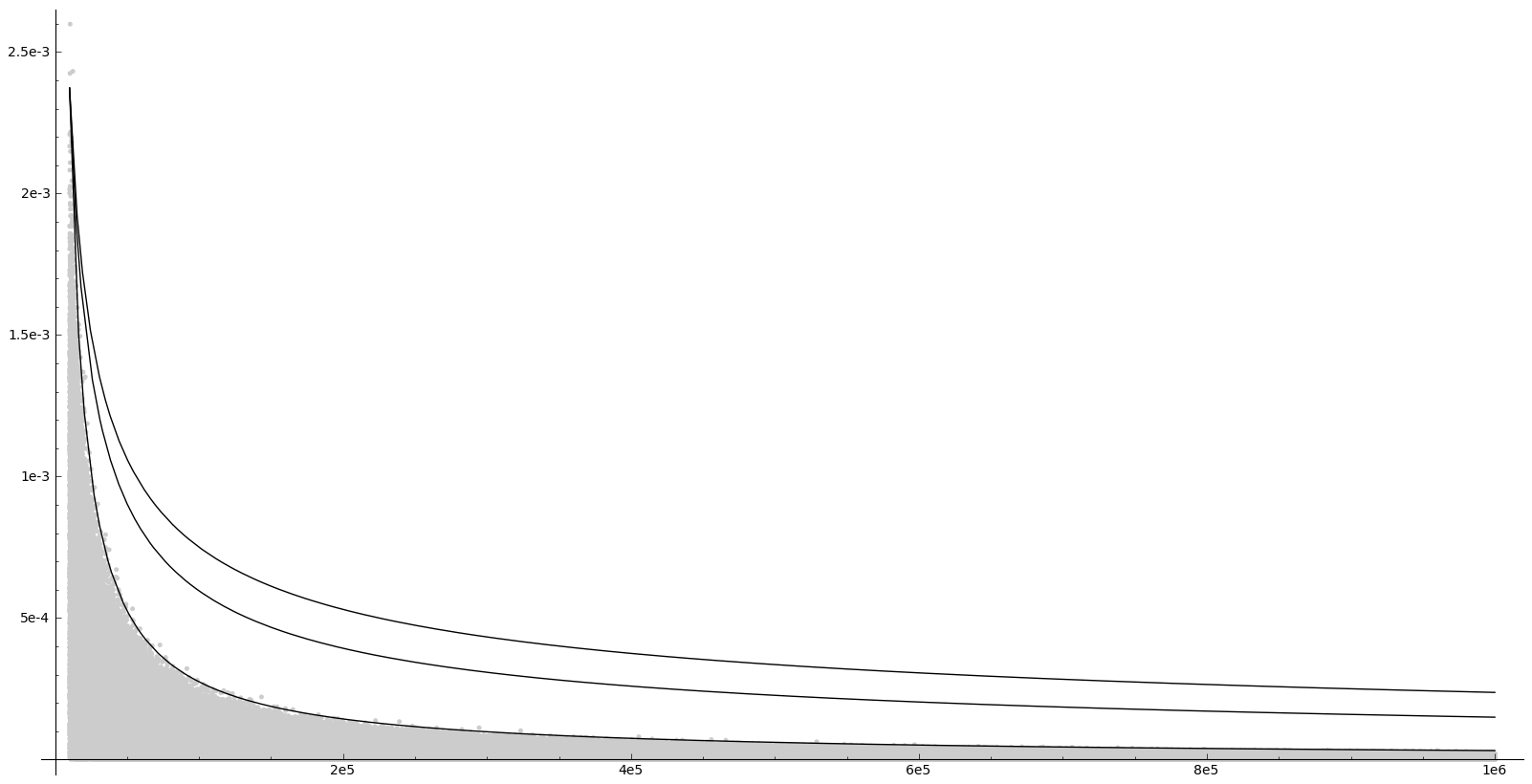}
\includegraphics[width=0.95\textwidth]{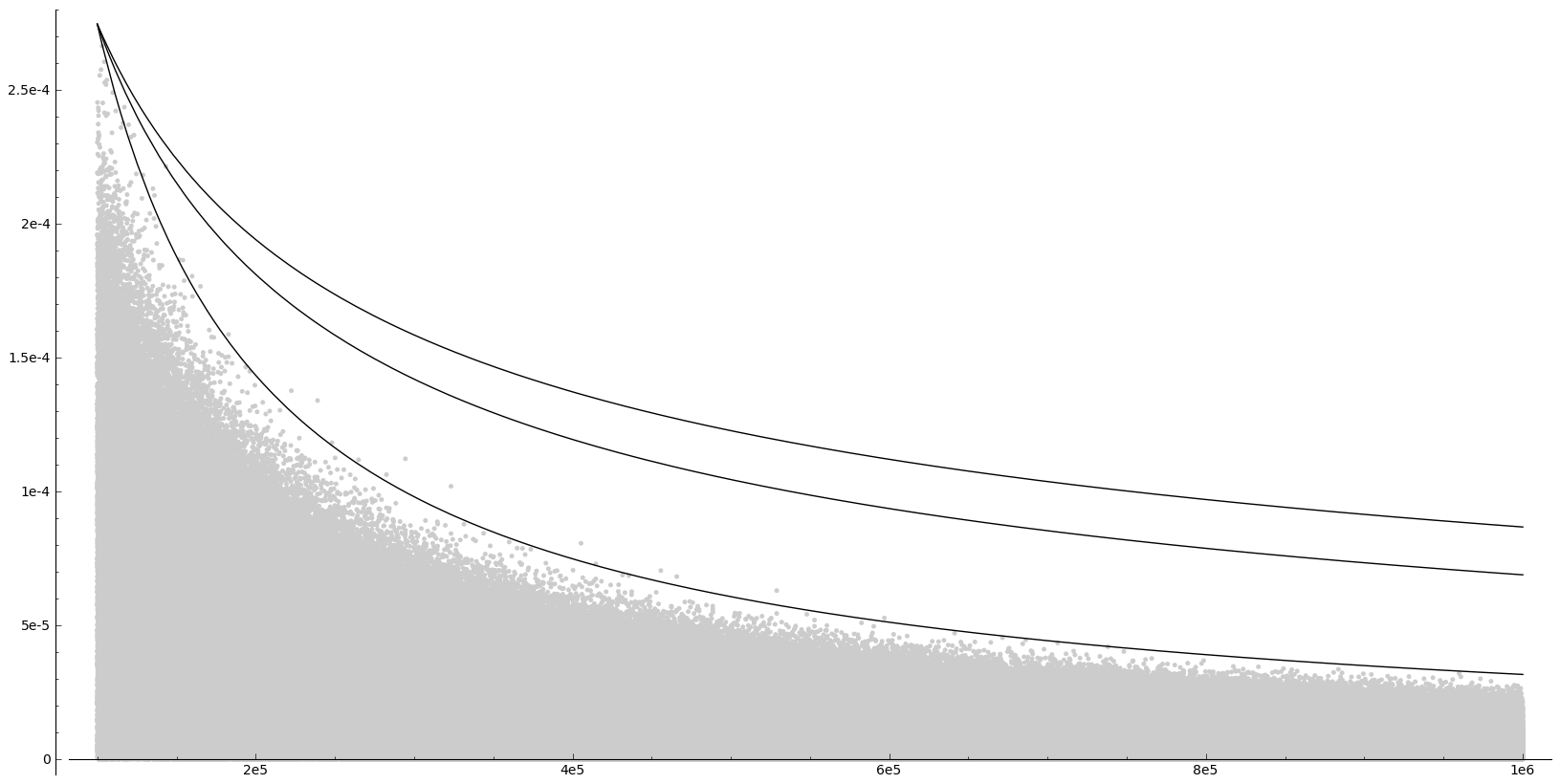}\\
The graphics include the points $(m,\Psi(m))$ for $m\in [\![10^4,10^6]\!]$ above and for $m\in [\![10^5,10^6]\!]$ below, curves of decreasing order $t^{-0.93}$, $t^{-3/5}$ and $t^{-1/2}$.
The curves are chosen so that they have the same value at the endpoint on the left.
\end{figure}

\begin{example}\label{ex1:Q=I_41}
Consider the $\R$-hermitian form, or `quadratic form', $Q=I_{4,1}=\left(\begin{smallmatrix}I_4&\\&-1  \end{smallmatrix}\right)$, where $I_4$ denotes the $4\times4$ identity matrix.
Here we have $C(Q,-1)=5$ (see Subsection~\ref{subsec:calc_C(Q,-k)}).
Furthermore, \eqref{eq1:upper_tau_R} says that $\lambda_1\geq 5/4$, thus $\tau=\nu_1+\rho_{4,\R} \leq \sqrt{(3/2)^2-5/4}+3/2= 5/2$ from \eqref{eq1:tau} (see also Corollary~\ref{cor2:cota_tau_R}).
Hence, we have
\begin{equation*}
    N_t(I_{4,1},-1)= 5\, t^{3} + O(t^{2.5})
    \qquad\text{as }t\to+\infty.
\end{equation*}
On the other hand, if there do not exist exceptional eigenvalues of the Laplace-Beltrami operator on $\Gamma_{Q}^0\ba \mathrm H_\R^4$ (or more generally \eqref{eq1:weaker_cond} holds), \eqref{eq1:tau_sin_autov_excep} ensures that
\begin{equation*}
    N_t(I_{4,1},-1)= 5\, t^{3} + O(t^{2.4})
    \qquad\text{as }t\to+\infty.
\end{equation*}

The formulas given above mean that $\Psi(t)\ll t^{-0.5}$ and $\Psi(t)\ll t^{-0.6}$ respectively.
On the other hand, our approximation of $\sigma$ is $-0.93$, more precisely, $\Psi(t)\sim 13\,t^{-0.93}$.
Figure~\ref{fig:Q=(I_4,-1)} shows the points $(m,\Psi(m))$ up to $10^6$ and the three different curves mentioned above.

We can see that the approximating curve is under the curve obtained by assuming $\lambda_1\geq\rho_{4,\R}^2=9/4$, thus these computations do not contradict the assumption of the nonexistence of exceptional eigenvalues in $\Gamma_Q^0\ba\mathrm H_\R^4$.
On the other hand, this numerical test gives an evidence for the lower bound $\lambda_1(\Gamma_Q^0\backslash \mathrm H_\R^4)\geq\rho_{4,\R}^2 \,4n/(n+1)^2=36/25=1.44$.
\end{example}

\begin{figure}
\caption{Graphic of $\Psi_t(Q,-1)$ when $\F=\C$, $\mathcal O=\Z[\sqrt{-3}]$, $H=I_{5,1}$ and $k=1$}
\label{fig:H=(I_5,-1)}
\includegraphics[width=0.95\textwidth]{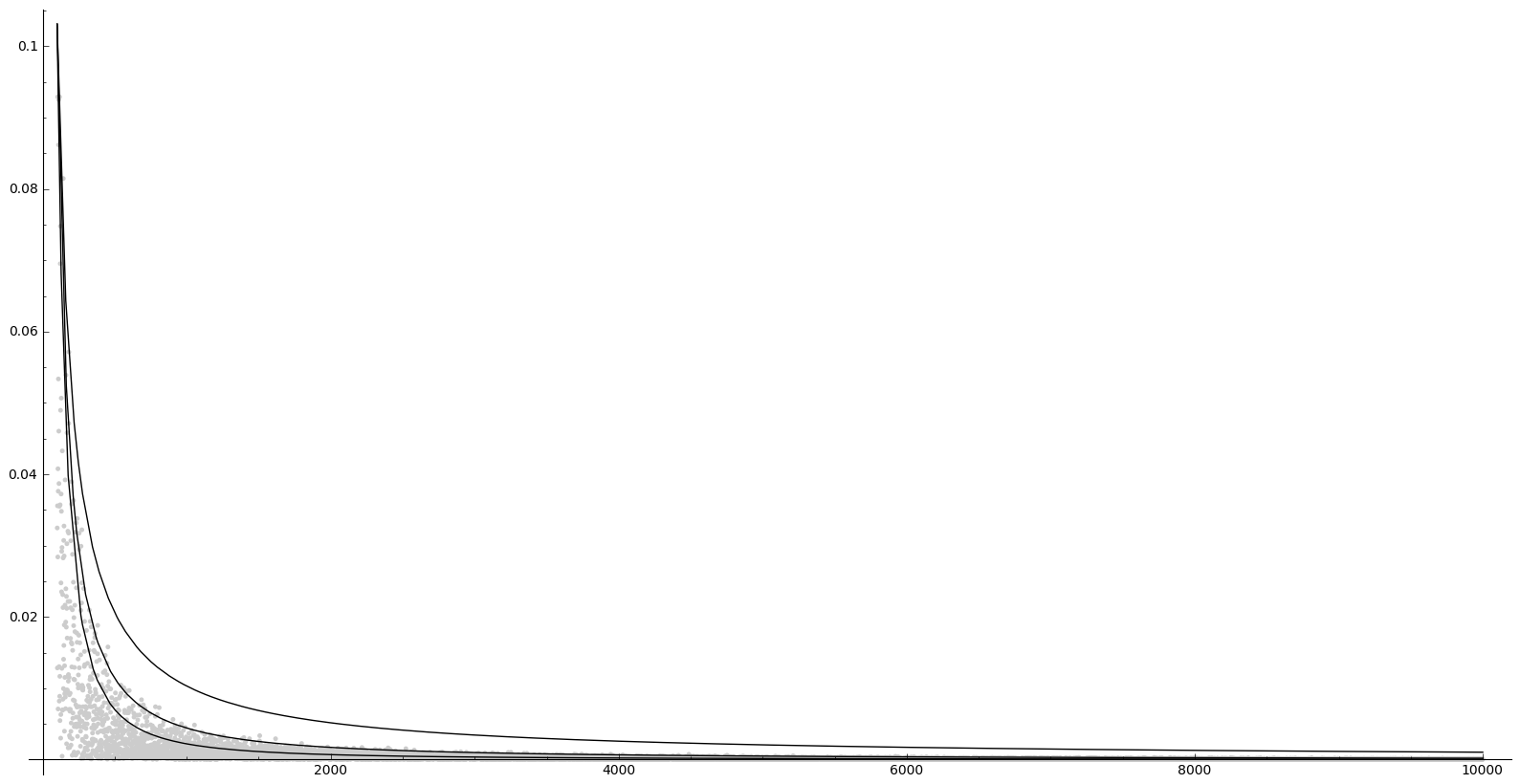}
\includegraphics[width=0.95\textwidth]{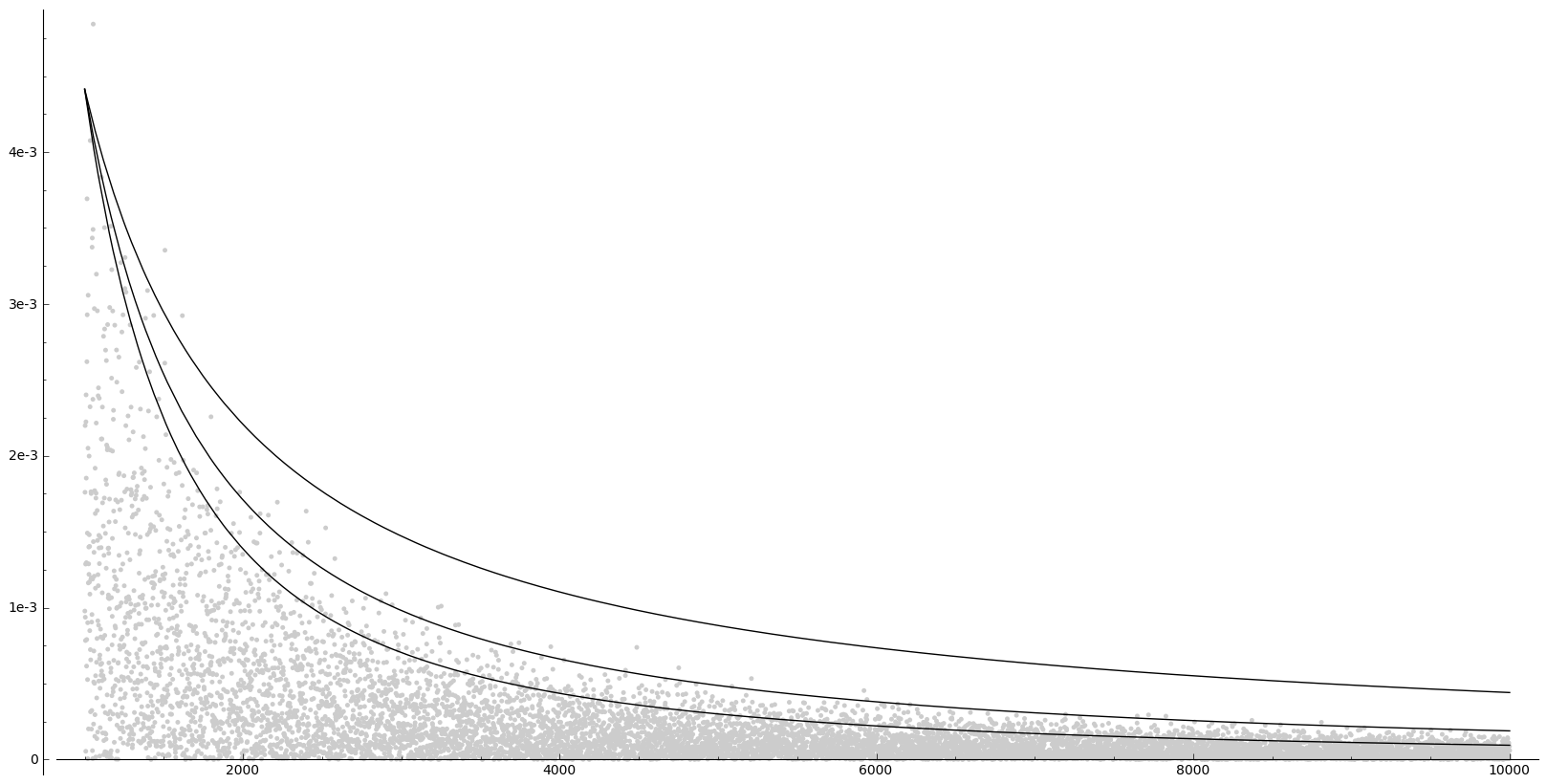}\\
The graphics include the points $(m,\Psi(m))$ for $m\in [\![10^2,10^4]\!]$ above and for $m\in [\![10^3,10^4]\!]$ below, curves of decreasing order $t^{-1.67}$, $t^{-1.36}$ and $t^{-1}$.
The curves are chosen so that they have the same value at the endpoint on the left.
\end{figure}

\begin{example}\label{ex1:H=I_51}
We now consider, for $\F=\C$ and $\mathcal O=\Z[\sqrt{-3}]$, the hermitian form $H=I_{5,1}=\left(\begin{smallmatrix}I_5&\\&-1  \end{smallmatrix}\right)$.
We have $C(H,-1)=18$ (see Subsection~\ref{subsec:calc_C(Q,-k)}), hence
\begin{equation*}
    N_t(H,-1)= 18\, t^{10} + O(t^{\tau})
    \qquad\text{as }t\to+\infty,
\end{equation*}
where $\tau$ is given by \eqref{eq1:tau}.

In this case, if we assume that there are no exceptional eigenvalues, or more generally \eqref{eq1:weaker_cond}, then the above formula holds with $\tau=50/6+\varepsilon$ ($50/6\approx8.33$) for each $\varepsilon>0$, hence $\Psi(t)\ll t^{-1.67}$.
However, this assumption contradicts our approximation $\Psi(t)\sim 56\, t^{-1.36}$, thus this is numerical evidence on the existence of exceptional spectrum in $\Gamma_{H}^0\ba \mathrm H_\C^5$.
Moreover, the estimation $\tau-2\rho\approx\sigma=-1.36$ tell us that $\lambda_1\approx 11.82$ since $\tau=\sqrt{5^2-\lambda_1}+5$ and $2\rho=10$.

Figure~\ref{fig:H=(I_5,-1)} shows the points $(m,\Psi(m))$ up to $10^4$ and the three different curves mentioned above.

In this case, Jian-Shu Li's lower bound \cite{Li} (see Theorem~\ref{thm2:lambda_1_C}) implies that $\lambda_1\geq 9$, where $\lambda_1$ denotes the first nonzero eigenvalue for the Laplace-Beltrami operator on the subspace of nondegenerate forms (nontrivial Fourier coefficients) in $L^2(\Gamma_H^0\ba \mathrm H_\C^n)$.
If we assume that this lower bound holds on all $L^2(\Gamma_H^0\ba \mathrm H_\C^5)$, we obtain that $\tau\leq 9$ (see Remark~\ref{rem2:cota_tau_C}), hence $\Psi(t)\in O(t^{-1})$.
Therefore, our estimation $\Psi(t)\sim 56\, t^{-1.36}$ give evidence on the extended Li's lower bound for $\Gamma_H^0$.
\end{example}

\section{Upper bounds for $\tau$}\label{sec:exceptional}

\subsection{Hyperbolic spaces}\label{subsec:hyp_spaces}

In this subsection we may assume that $\F=\R$, $\C$ or $\mathbb H$.
It contains a brief review of Sections 2 and 3 in \cite{Lauret12}.
We consider $\F^{n+1}$ as a right module over $\F$.
Let $Q$ be an $\F$-hermitian matrix as in \eqref{eq1:Q}.
This matrix induces the $\F$-sesquilinear form $Q(x,y):=x^*Qy$ and the $\F$-hermitian form $Q[x]:=Q(x,x)=x^*Qx$.

We consider the $n$-dimensional $\F$-hyperbolic space realized by the \emph{$Q$-Kleinian model}
\begin{equation}
  \mathrm H_\F^n(Q) = \left\{ [x]\in\mathrm{P}\F^{n}: Q[x]<0 \right\}.
\end{equation}
Under a standard Riemannian structure described in \cite[\S2]{Lauret12}, the distance between two points $[x]$ and $[y]$ in $\mathrm H_\F^n(Q)$ is given by
\[
\cosh(d([x],[y])) = \frac{|Q(x,y)|}{|Q[x]|^{1/2}\,|Q[y]|^{1/2}},
\]
Let $\U(Q,\F)$ be the \emph{$Q$-unitary group} and let $\SU(Q,\F)$ be the \emph{special $Q$-unitary group}, that is
\begin{align*}
  \U(Q,\F) =& \, \{g\in\GL(n+1,\F):Q[g]=g^*Qg=Q\},\\
  \SU(Q,\F)=& \, \{g\in\U(Q,\F):\det(g)=1\}.
\end{align*}
The action of these groups on $\mathrm H_\F^n(Q)$ is given by $g\cdot [x]=[gx]$ for $g\in\U(Q,\F)$.
This action is by isometries, transitive and the only elements that induces the identity map are those in the center $Z(\U(Q,\F))$ of $\U(Q,\F)$.

Let $G=\SU^0(Q,\F)$ be the identity component of $\SU(Q,\F)$.
When $\F=\C$ or $\mathbb H$, the group $\SU(Q,\F)$ is already connected, thus $G=\SU(Q,\F)$.
However, if $\F=\R$ we have
\begin{equation*}
G=\PSO(Q):=\{g\in\U(Q,\R):g_{n+1,n+1}>0\}.
\end{equation*}
In all cases, $\mathrm H_\F^n(Q)$ coincides with the symmetric space $G/K$, where $K$ is the isotropy subgroup in $G$ of $[e_{n+1}]\in\mathrm H_\F^n(Q)$ (see \cite[\S3]{Lauret12}).

Let $\Gamma_Q=\U(Q,\F)\cap\M(n+1,\mathcal O)$ denote the unimodular matrices in $\U(Q,\F)$, and let $\Gamma_Q^0=\Gamma_Q\cap G$.
It is a simple matter to see that
\[
[\Gamma_Q:\Gamma_Q^0]=
\begin{cases}
  4&\text{ if $\F=\R$},\\
  \omega&\text{ if $\F=\C$},
\end{cases}
\]
where $\omega$ is the number of units in $\mathcal O$.
These groups act discontinuously on $\mathrm H_\F^n(Q)$ and the quotients $\Gamma_Q\backslash \mathrm H_\F^n(Q)$ and $\Gamma_Q^0\backslash \mathrm H_\F^n(Q)$ have finite volume.
Moreover, $\Gamma_Q^0\backslash \mathrm H_\F^n(Q)$ is compact if and only if the $\F$-hermitian form $Q$ is anisotropic over the quotient field of $\mathcal O$.

\subsection{Exceptional eigenvalues}\label{subsec:upper_bounds}

Let $\Delta$ be the Laplace-Beltrami operator on $\mathrm H_\F^n(Q)$.
For a general discrete non-cocompact subgroup of finite covolume of $G$,
the operator
\[
-\Delta:L^2(\Gamma\backslash \mathrm H_\F^n(Q))\longrightarrow L^2(\Gamma\backslash \mathrm H_\F^n(Q))
\]
is known to be essentially self-adjoint and positive and one has the spectral decomposition
\[
L^2(\Gamma\backslash \mathrm H_\F^n(Q)) = L^2_d(\Gamma\backslash \mathrm H_\F^n(Q)) \oplus^{\bot} L^2_c(\Gamma\backslash \mathrm H_\F^n(Q)),
\]
where the spectrum of $\Delta$ is discrete (resp.\ continuous) in $L^2_d(\Gamma\backslash \mathrm H_\F^n(Q))$ (resp.\ $L^2_c(\Gamma\backslash \mathrm H_\F^n(Q))$).
Let $\rho=(n+1)r/2-1$ where $r=\dim_\R(\F)$.
There are only finitely many eigenvalues of the discrete spectrum of $\Delta$ (counted with multiplicities) such that
\[
0<\lambda_1\leq\dots\leq\lambda_N<\rho^2.
\]
These are called \emph{exceptional eigenvalues}.

\begin{definition}\label{def2:cong-subg}
Let $Q$ be an $\F$-hermitian form with coefficients in the quotient field of $\mathcal O$ with signature $(n,1)$.
A subgroup of $\Gamma_Q^0$ is called a \emph{congruence subgroup} if it contains
\[
\Gamma_Q^0(\mathfrak a):=
\left\{
g\in \Gamma_Q^0: g\equiv I_{n+1}\pmod{\mathfrak a}
\right\}
\]
for some ideal $\mathfrak a$ of $\mathcal O$.
\end{definition}

We recall that the error term in the main formula of Theorem~\ref{thm1:main_form} is $O(t^\tau)$ where $\tau=\tau(\Gamma_Q^0)$ is given by \eqref{eq1:tau} associated to the discrete subgroup $\Gamma_Q^0$.
Hence, a lower bound for $\lambda_1$ makes the error term smaller.
Our goal in this section is to apply known lower bound for $\lambda_1(\Gamma\ba\mathrm H_\F^n)$ with $\Gamma$ a congruence subgroup, obtaining upper bound for $\tau$.

We begin in the hyperbolic plane case (i.e.\ $\F=\R$ and $n=2$).
The following result is due to Kim and Sarnak~\cite{Kim-Sarnak}.

\begin{theorem}\label{thm2:lambda_1n=2}
For any congruence subgroup $\Gamma<\SL(2,\Z)$, the first nonzero eigenvalue $\lambda_1$ for the Laplace-Beltrami operator on $\Gamma\backslash \mathrm H_\R^2$ satisfies $\lambda_1\geq975/4096 \approx 0.238\dots\,.$
\end{theorem}

In Introduction we mentioned that, for any $n\geq4$, there are congruence subgroups such that their associated hyperbolic manifolds have exceptional spectrum (see \cite[\S6]{Cogdell-Li-Piatetski-Shapiro-Sarnak}).
The following result is the optimal lower bound for $n\geq4$.
It was proved independently by Elstrodt, Grunewald and Mennicke~\cite[Thm.~A]{EGMKloosterman} and Cogdell, Li, Piatetski-Shapiro and Sarnak~\cite{Cogdell-Li-Piatetski-Shapiro-Sarnak}.

\begin{theorem}\label{thm2:lambda_1n>2}
  Let $n\geq3$ and let $Q$ be a quadratic form with rational coefficients such that $Q$ is of signature $(n,1)$ and isotropic over $\Q$.
  For any congruence subgroup $\Gamma<\Gamma_Q^0$, the first nonzero eigenvalue $\lambda_1$ of the Laplace-Beltrami operator on $\Gamma\backslash \mathrm H_\R^n(Q)$ satisfies
  \[
  \lambda_1\geq\dfrac{2n-3}{4}
  \]
\end{theorem}

\begin{corollary}\label{cor2:cota_tau_R}
  Let $\F=\R$, let $\Gamma$ be a congruence subgroup acting on the $n$-dimensional real hyperbolic space, and let $\tau=\tau(\Gamma)$ given by \eqref{eq1:tau} associated to $\Gamma$.
  Then
  \begin{equation}\label{eq2:tau_R}
  \tau\leq\begin{cases}
    n-3/2&\text{ if $n\geq3$,}\\[2mm]
    2/3 &\text{ if $n=2$.}
  \end{cases}
  \end{equation}
\end{corollary}

\begin{proof}
We recall that $\nu_j:=\sqrt{\rho^2-\lambda_j}$.
For $n\geq3$, Theorem~\ref{thm2:lambda_1n>2} imply that
\[
\nu_1 \leq \sqrt{\left(\frac{n-1}2\right)^2 -\frac{2n-3}4}=\frac{n-2}2.
\]
From \eqref{eq1:tau} we conclude that $\tau \leq \nu_1+\rho\leq (n-1)/2+(n-2)/2=n-3/2$.

For $n=2$, Theorem~\ref{thm2:lambda_1n=2} implies that \eqref{eq1:weaker_cond} holds since
\[
\frac{4n\,\rho^2}{(n+1)^2}=\frac29=0.22\dots<\frac{975}{4096}\leq\lambda_1,
\]
then $\tau=2\rho\, n/(n+1)=2/3$ by \eqref{eq1:tau}.
\end{proof}

\begin{remark}
Note that Corollary~\ref{cor2:cota_tau_R} gives the best error term for \eqref{eq1:main-form} that can be obtained from \cite{Lauret12}, i.e.\ having used the lattice point theorem of Levitan~\cite{Levitan}.
\end{remark}

Bounds of this kind are not known in the complex case.
A related result was proved by Li~\cite[Cor.~1.4]{Li}.

\begin{theorem}\label{thm2:lambda_1_C}
Let $n\geq2$ and let $\Gamma$ be a non-cocompact congruence subgroup of $\SU(n,1)$.
The first nonzero eigenvalue $\lambda_1$ for the Laplace-Beltrami operator on the space of nondegenerate forms in $L^2(\Gamma\backslash \mathrm H_\C^n(Q))$ satisfies
\[
\lambda_1\geq 2n-1.
\]
\end{theorem}

\begin{remark}\label{rem2:cota_tau_C}
The analogous procedure to that in Corollary~\ref{cor2:cota_tau_R} in this case, by assuming that Theorem~\ref{thm2:lambda_1_C} holds for every form in $L^2(\Gamma\backslash \mathrm H_\C^n(Q))$, gives that $\nu_1\leq n-1$, hence
\begin{equation}\label{eq2:tau_C}
\tau\leq \nu_1+\rho= n-1+n=2n-1.
\end{equation}
\end{remark}

\section{Calculation of local densities}\label{sec:local_densities}
The purpose in this section is to compute the \emph{local density} $\delta(Q,-k)$ of the representation $Q[x]=-k$.
Here we assume $\F=\R$ or $\C$.
Recall that $r=\dim_\R(\F)$, $\mathcal O$ is $\Z$ in the real case and the ring of integers of an imaginary quadratic number field if $\F=\C$, and $d_\mathcal O$ denotes the discriminant of the quotient field of $\mathcal O$.

For an arbitrary nondegenerate $\F$-hermitian forms and any $k\in\Z$, the local density of the representation $Q[x]=k$ is defined by
\begin{equation}\label{eq3:density}
\delta(Q,k) = \prod_p \;\delta_p(Q,k),
\end{equation}
where the product runs over all positive prime numbers and
\begin{equation}\label{eq3:p-denstity}
\delta_p(Q,k):=
\lim_{j\to\infty} \frac{\# \left\{x\in{\big(\mathcal O /p^j\mathcal O\big)}^{m}: Q[x]\equiv k\;\bmod p^j \right\}}{p^{j(r m-1)}}.
\end{equation}

Throughout this section, we only consider diagonal integral nondegenerate $\F$-hermitian forms of rank $m$, thus, they are given by matrices of the form
\begin{equation}\label{eq3:Q_diagonal}
\begin{pmatrix}
  a_1\\ &\ddots\\ &&a_m
\end{pmatrix}
  \qquad (a_i\in\Z,\;a_i\neq0\;\forall i).
\end{equation}
Note that the signature of this matrix is $(m_1,m_2)$ where $m_1$ (resp.\ $m_2$) denotes the number of positive (resp.\ negative) entries in the diagonal.

As usual, let $\zeta(s)$ denote the Riemann zeta function.
We shall use the well known formula
\begin{equation}\label{eq3:zeta(m)}
\zeta(m)=\frac{(2\pi)^m|B_m|}{2\, m!},
\end{equation}
for every $m$ odd integer, where $B_m$ denotes the $m$-th Bernoulli number

We will denote by $\legendre{\cdot}{m}$ the Kronecker symbol.
If $D$ is a quadratic discriminant (i.e.\ $D\equiv1\pmod4$ square free or $D\equiv0\pmod4$ with $D/4\equiv 2,3\pmod4$ square free), let $L(s,D)$ denote the Dirichlet $L$-series defined by
\[
L(s,D)=\sum_{m=1}^\infty\legendre{D}{m}\,m^{-s}=\prod_{p}\left( 1-\legendre{D}{p}\,p^{-s} \right)^{-1}.
\]

\subsection{Real case}
Suppose $\F=\R$, thus $\mathcal O=\Z$, $d_\mathcal O=1$ and $r=1$.
In this subsection we shall use the terminology `quadratic form' instead $\R$-hermitian form.
Every $p$-local density $\delta_p(Q,k)$ has been computed by Yang~\cite{Yang} for every integral nondegenerate quadratic form $Q$ and any $k\in\Z$.
We will use his results several times, as in the following lemma.

\begin{lemma}\label{lem3:delta_pR}
  Let $Q$ be a quadratic form as in \eqref{eq3:Q_diagonal}, let $k\in\Z$ and let $p$ be a positive prime number such that $p\nmid 2\det(Q)$.
  Write $k=p^c\ell$ with $p\nmid \ell$.
  Set $q=p^{-\frac{m-2}2}$ and
  \[
  \epsilon =
    \begin{cases}
      \legendre{(-1)^{\frac m2}\, \det(Q)}{p} & \text{if $m$ is even},\\[3mm]
      \legendre{(-1)^{\frac{m-1}2}\, \det(Q)\, \ell}{p} &\text{if $m$ is odd}.
    \end{cases}
  \]
  Then
\begin{equation}\label{eq3:delta_p(Q,k)}
  \delta_p(Q,k) =
  \begin{cases}
    \left(1-\epsilon p^{-\frac{m}2}\right) \; \dfrac{\left(\epsilon q\right)^{c+1}-1}{\epsilon q-1}
        &\;\text{if $m$ is even,}\\[4mm]
    \left(1- p^{-(m-1)}\right) \; \dfrac{q^{c+1}-1}{q^2-1}
        &\;\text{if $m$ is odd and $c$ is odd,}\\[4mm]
    \left(1-p^{-(m-1)}\right) \; \dfrac{q^c-1}{q^2-1}  +  (1+\epsilon p^{-\frac{m-1}2})q^c
        &\;\text{if $m$ is odd and $c$ is even.}
  \end{cases}
\end{equation}
In particular, if $p\nmid k$, then
\begin{equation}\label{eq3:delta_p(Q,k)_k}
  \delta_p(Q,k) =
  \begin{cases}
    1-\legendre{(-1)^{\frac{m}2} \det(Q)}{p} p^{-\frac{m}2}
        &\;\text{if $m$ is even,}\\[2mm]
    1+\legendre{(-1)^{\frac{m-1}2} \det(Q)\, \ell}{p} p^{-\frac{m-1}2}
        &\;\text{if $m$ is odd.}
  \end{cases}
\end{equation}
\end{lemma}

\begin{proof}
By assumption, $p\neq2$ and $p\nmid a_i$ for every $i$, thus $l_i=0$ and $\varepsilon_i=a_i$ for all $i$, in the notation of \cite{Yang}, hence $L(j,1)=[1,m]$ for $j$ even and $L(j,1)=\emptyset$ for $j$ odd.
If $m$ is even, \cite[Theorem~3.1]{Yang} implies
\begin{align*}
\delta_p(Q,k)
    &= 1+(1-p^{-1}) \sum_{j=1}^c \epsilon^j q^{j} - \epsilon^{c+1} q^{c+1} p^{-1}  \\
    &= 1-(\epsilon q)^{c+1}+(1-p^{-1}) \,\sum_{j=1}^{c+1} \left(\epsilon q\right)^{j}\\
    &= \frac{\left(\epsilon q\right)^{c+1}-1}{\epsilon q-1}\;\left(1-\epsilon p^{-\frac{m}2}\right).
\end{align*}
The proof for even $n$ is similar.
\end{proof}

The cases when $p|2\det(Q)$ can also be computed from \cite{Yang}.

\begin{theorem}\label{thm3:real-density}
  Let $Q$ be a quadratic form as in \eqref{eq3:Q_diagonal} and let $k\in\Z$.
  If $m$ is odd, write $k\det(Q)=c^2 b$ ($c,b\in\Z$) with $b$ square free and set
  \[
  D=\begin{cases}
     (-1)^{\frac{m-1}2} b & \text{if }(-1)^{\frac{m-1}2}b \equiv1\pmod4,\\
    4(-1)^{\frac{m-1}2} b & \text{otherwise.}
  \end{cases}
  \]
  Then
  \begin{equation}\label{eq3:delta(Q,-k)_n_even}
  \delta(Q,k)=
    \frac{L(\frac{m-1}2,D)}{\zeta(m-1)}\;\prod_{p\mid D}
    \frac{\delta_p(Q,k)}{1-p^{-(m-1)}}\prod_{p\nmid D \atop p\mid 2k\det(Q)}
    \frac{\delta_p(Q,k)}{1+\legendre{D}{p} p^{-\frac{m-1}2}}.
  \end{equation}

  If $m$ is even, write $\det(Q)=c^2 b$ ($c,b\in\Z$) with $b$ square free and set
  \[
  D=\begin{cases}
     (-1)^{\frac{m}2} b & \text{if } (-1)^{\frac{m}2} b\equiv1\pmod4,\\
    4(-1)^{\frac{m}2} b & \text{otherwise.}
  \end{cases}
  \]
  Then
  \begin{equation}\label{eq3:delta(Q,-k)_n_odd}
  \delta(Q,k) =
    \frac{1}{L(m/2,D)}\;
    \prod_{p\mid 2k\det(Q)} \frac{\delta_p(Q,k)}{1-\legendre{D}{p}p^{-\frac{m}2}}.
  \end{equation}
\end{theorem}
The proof is very similar to the proof of \cite[Thm.~12]{Ratcliffe-Tschantz}.
It follows by decomposing $\delta(Q,k)$ as $\prod_{p\mid 2k\det(Q)}\delta_p(Q,k)\;\prod_{p\nmid \,2k\det(Q)}\delta_p(Q,k)$ and by applying Lemma~\ref{lem3:delta_pR}.
We will show the details in the complex case.

\subsection{Complex case}
Suppose now $\F=\C$, thus $r=2$ and $\mathcal O=\Z[\omega]$ where $\omega=\sqrt{d_{\mathcal O}}/2$ if $d_{\mathcal O}\equiv 0\pmod4$ or $\omega=(1+\sqrt{d_{\mathcal O}})/2$ if $d_{\mathcal O}\equiv1\pmod4$.

In this subsection we shall use the terminology `hermitian form' instead of $\C$-hermitian form, and we denote the form by $H$ (instead $Q$).
Actually, in the next lemma we introduce a quadratic form $Q$ associated to $H$.

\begin{lemma}\label{lem3:delta_H_C}
Let $H$ be a hermitian form over $\mathcal O$ as in \eqref{eq3:Q_diagonal} of signature $(m_1,m_2)$.
Set
\begin{equation}\label{eq3:H-Q}
Q_H=
\begin{cases}
\diag\left(
a_1,-a_1\frac{d_\mathcal O}4, \dots,a_m,-a_m\frac{d_\mathcal O}4
\right)
    &\text{ if }d_\mathcal O\equiv 0\pmod4,\\[2mm]
\diag\left(
\left(\begin{smallmatrix}a_1&a_1/2\\ a_1/2&a_1(1-d_\mathcal O)/4\end{smallmatrix} \right)
\dots
\left(\begin{smallmatrix}a_m&a_m/2\\ a_m/2&a_m (1-d_\mathcal O)/4\end{smallmatrix} \right)
\right)
    &\text{ if }d_\mathcal O\equiv 1\pmod4.
\end{cases}
\end{equation}
Then $Q_H$ is a quadratic form of signature $(2m_1,2m_2)$ (of rank $2m$) and
\[
\delta_p^\C(H,k) = \delta_p^\R(Q_H,k)
\]
for every $k\in\Z$, where the upper index $\F$ indicates the fields of definition of the forms.
\end{lemma}

\begin{proof}
We have
\[
H[z]=a_1\,|z_1|^2 + \dots + a_{n+1}\,|z_{n+1}|^2,\qquad \text{for }\; z\in\mathcal O^{n+1}.
\]
Writing  $z_j=x_{2j-1}+x_{2j}\,\omega$ with $x_{2j-1},x_{2j}\in\Z$ for each $z_j\in\mathcal O$, it follows that $H[z]$ induces the quadratic forms in \eqref{eq3:H-Q} for $x\in\Z^{2n+2}$.
Hence
\begin{align*}
\delta_p^\C(H,k)
    &=\lim_{j\to\infty} \frac{\# \left\{z\in{\big(\mathcal O /p^j\mathcal O\big)}^{m}: H[z]\equiv k\;\bmod p^j \right\}}{p^{j(2m-1)}}\\[3mm]
    &=\lim_{j\to\infty} \frac{\# \left\{x\in{\big(\Z /p^j\Z\big)}^{2m}: Q_H[x]\equiv k\;\bmod p^j \right\}}{p^{j(2m-1)}}
    =\delta_p^\R(Q_H,k),
\end{align*}
since a representative set of $\mathcal O/p^j\mathcal O$ is $\{b_1+b_2\omega:b_i\in\Z,\; 0\leq b_1,b_2<p^j\}$.
\end{proof}

\begin{remark}\label{rem3:delta-p=2}
When $p\neq2$, the matrices $\left(\begin{smallmatrix} 1&1/2\\ 1/2& (1-d_{\mathcal O})/4 \end{smallmatrix}\right)$ and $\left(\begin{smallmatrix} 1&\\ & -d_{\mathcal O}/4 \end{smallmatrix}\right)$ are equivalent over the $p$-adic integer numbers $\Z_p$ since
\[
\begin{pmatrix}1&-1/2\\ 0&1\end{pmatrix}^t
\begin{pmatrix} 1&1/2\\ 1/2& (1-d_{\mathcal O})/4 \end{pmatrix}
\begin{pmatrix}1&-1/2\\ 0&1\end{pmatrix} =
\begin{pmatrix}1&\\&-d_{\mathcal O}/4\end{pmatrix}.
\]
Consequently, both quadratic forms given in \eqref{eq3:H-Q} are equivalent over $\Z_p$ if $p\neq2$, thus their $p$-local densities representing any fixed $k$ coincide.
\end{remark}

\begin{lemma}
Let $H$ be a hermitian form over $\mathcal O$ as in \eqref{eq3:Q_diagonal}, let $k\in\Z$ and let $p$ be a positive prime number such that $p\nmid 2d_{\mathcal O} \det(H)$.
Write $k=p^c\ell$ with $p\nmid \ell$ and set $\epsilon=\legendre{d_{\mathcal O}}{p}^{m}$ and $q=p^{-(m-1)}$.
Then
\begin{equation}\label{eq3:delta_p(H,k)}
  \delta_p(H,-k) = \left(1-\epsilon p^{-m}\right) \frac{(\epsilon q)^{c+1}-1}{\epsilon q-1}.
\end{equation}
In particular, if $p\nmid k$, then
\begin{equation}\label{eq3:delta_p(H,k)_k}
  \delta_p(H,-k) = 1-\legendre{d_{\mathcal O}}{p}^{m} p^{-m}.
\end{equation}
\end{lemma}

\begin{proof}
From Lemma~\ref{lem3:delta_H_C} we see that $\delta_p(H,k)=\delta_p^\R(Q_H,k)$ where $Q_H$ is given by \eqref{eq3:H-Q}.
Moreover, since $p\neq2$, Remark~\ref{rem3:delta-p=2} tell us that $\delta_p(H,k)=\delta_p^\R(Q_H',k)$ where
\[
Q_H'=\diag\left(
a_1,-a_1\tfrac{d_\mathcal O}4, \dots,a_m,-a_m\tfrac{d_\mathcal O}4
\right).
\]

We now compute $\delta_p^\R(Q_H',k)$ by applying \cite{Yang}.
In its notation, since $p\nmid d_\mathcal O$ and $p\nmid a_i$ for all $i$, we have $l_i=0$ for all $1\leq i\leq 2m$, $\epsilon_{2j-1}=a_j$ and $\epsilon_{2j}=-a_jd_{\mathcal O}/4$ for $1\leq j\leq m$, hence $L(j,1)=[1,2m]$ for $j$ odd and $L(j,1)=\emptyset$ for $j$ even.
In particular, $l(j,1)$ is even for any $j$.
Furthermore, $d(j)=-j(m-1)$ and $v(j)=\legendre{d_\mathcal O}{p}^{jm}=\epsilon^j$.
Theorem~3.1 in \cite{Yang} now yields that
\begin{align*}
\delta(Q,k)
    &= 1+(1-p^{-1})\sum_{j=1}^c (\epsilon q)^j- (\epsilon q)^{c+1} p^{-1}.
\end{align*}
The rest of the proof is straightforward.
\end{proof}

\begin{theorem}\label{thm3:complex-density}
Let $H$ be a hermitian form as in \eqref{eq3:Q_diagonal} over $\mathcal O$ and let $k\in\Z$.
Write $k=\pm\prod_p p^{\tau_p}$ with $\tau_p\in\N\cup\{0\}$ and $\tau_p=0$ for almost all $p$.

If $m$ is odd, then
\begin{multline}\label{eq3:delta(H,k)_m_odd}
\delta(H,k)
    =  \frac{1}{L(m,d_{\mathcal O})}
        \prod_{p\mid 2d_\mathcal O\det(H)}  \frac{\delta_p(H,k)}{1-\legendre{d_{\mathcal O}}{p} p^{-m}}\;\times
        \prod_{p\mid k \atop p\nmid 2 d_{\mathcal O} \det(H) }  \frac{1-\legendre{d_{\mathcal O}}{p}^{\tau_p+1}p^{-(\tau_p+1)(m-1)}}{1-\legendre{d_{\mathcal O}}{p}p^{-(m-1)}}.
\end{multline}

If $m$ is even, then
\begin{align}\label{eq3:delta(H,k)_m_even}
\delta(H,k)
    &= \frac{1}{\zeta(m)} \prod_{p\mid 2 d_{\mathcal O} \det(H) } \frac{\delta_p(H,k)}{1-p^{-m}}  \prod_{p\mid k \atop p\nmid 2 d_{\mathcal O} \det(H)} \frac{1-p^{-(\tau_p+1)(m-1)}}{1-p^{-(m-1)}}.
\end{align}
\end{theorem}

\begin{proof}
From \eqref{eq3:delta_p(H,k)_k} we have
\begin{align*}
\delta(H,k)
    &=  \prod_{p\mid  2kd_{\mathcal O}\det(H)} \delta_p(H,k)
        \prod_{p\nmid 2kd_{\mathcal O}\det(H)} {\textstyle\left(1-\legendre{d_{\mathcal O}}{p}^{m}p^{-m}\right)}.
\end{align*}
Recall that $\legendre{d_{\mathcal O}}{p}=\pm1$ if $p\nmid d_{\mathcal O}$.
Thus, if $m$ is even, we obtain that
\begin{align*}
\delta_p(H,k)
    &=  \prod_{p\mid  2kd_{\mathcal O}\det(H)} \delta_p(H,k)
        \prod_{p\nmid 2kd_{\mathcal O}\det(H)} \left(1-p^{-m}\right) \\
    &=  \frac{1}{\zeta(m)}
        \prod_{p\mid  2kd_{\mathcal O}\det(H)} \frac{\delta_p(H,k)}{1-p^{-m}} \\
    &=  \frac{1}{\zeta(m)}
        \prod_{p\mid  2d_{\mathcal O}\det(H)} \frac{\delta_p(H,k)}{1-p^{-m}}
        \prod_{p\mid k \atop p\nmid  2d_{\mathcal O}\det(H)} \frac{\delta_p(H,k)}{1-p^{-m}}.
\end{align*}
Hence, \eqref{eq3:delta(H,k)_m_even} follows from \eqref{eq3:delta_p(H,k)}.

On the other hand, if $n$ is even, we have that
\begin{align*}
\delta_p(H,k)
    &=  \prod_{p\mid  2kd_{\mathcal O}\det(H)} \delta_p(H,k)
        \prod_{p\nmid 2kd_{\mathcal O}\det(H)} \left(1-\legendre{d_{\mathcal O}}{p} p^{-m}\right) \\
    &=  \frac{1}{L(n+1,d_{\mathcal O})}
        \prod_{p\mid  2d_{\mathcal O}\det(H)} \frac{\delta_p(H,k)}{1-\legendre{d_{\mathcal O}}{p} p^{-m}}
        \prod_{p\mid k \atop p\nmid 2d_{\mathcal O}\det(H)} \frac{\delta_p(H,k)}{1-\legendre{d_{\mathcal O}}{p} p^{-m}}.
\end{align*}
Finally, \eqref{eq3:delta(H,k)_m_odd} follows from \eqref{eq3:delta_p(H,k)}.
\end{proof}

\section{Numerical approximation of the error}\label{sec:approx}

The goal of this section is to approximate the behavior of
\begin{equation}\label{eq4:Psi(t)}
  \Psi(t)=\Psi_t(Q,-k)= \left|\frac{N_t(Q,-k)}{t^{2\rho}} - \coef\right|
\end{equation}
for $t\leq T$, where $T=10^6,\,10^4$ if $\F=\R,\,\C$ respectively.
Here and subsequently, we consider $k\in\N$ and $Q=\left(\begin{smallmatrix}A\\&-a\end{smallmatrix}\right)$ an $\F$-hermitian form as in \eqref{eq1:Q} where $a\in\N$ and $A$ is diagonal with positive integer entries $a_1,\dots,a_n$, that is
\begin{equation}\label{eq4:Q_diag_sign(n,1)}
Q=\begin{pmatrix}
  a_1\\ &\ddots\\ &&a_n\\ &&&-a
\end{pmatrix}.
\end{equation}

This section is divided into three parts: calculation of $C(Q,-k)$, calculation of $N_t(Q,-k)$ for every $t\leq T$ and finally, approximation of the behavior of $\Psi(t)$.

\subsection{Calculation of $\coef$}\label{subsec:calc_C(Q,-k)}
By writing the main coefficient \eqref{eq1:C(Q,-k)} as $\coef=\coeff \delta(Q,-k)$, $\coeff$ depend only on the $\F$-hermitian form $Q$. If $Q$ is as in \eqref{eq4:Q_diag_sign(n,1)} then
\begin{equation}\label{eq4:C'(Q)}
\coeff=
\begin{cases}
    \displaystyle \frac{2\,\pi^{n/2}}{(n-1)\Gamma(\frac{n}2)}
    \; \frac{a^{\frac{n-2}{2}}}{(a_1\dots a_{n})^{\frac{1}{2}}}
        &\text{ if $\F=\R$,}\\[7mm]
    \displaystyle \frac{2^{n+1}\,\pi^{n+1}}{n!\,|d_{\mathcal O}|^{\frac{n+1}2}}\;
    \frac{a^{n-1}}{a_1\dots a_{n}}
        &\text{ if $\F=\C$,}\\
\end{cases}
\end{equation}
since $r=\dim_\R(\F)$, $2\rho=n-1,\,2n$ if $\F=\R,\C$ respectively, and $\mathrm{vol}(S^{m-1})= 2\pi^{m/2}/\Gamma(m/2)$.

Furthermore, the local density $\delta(Q,-k)$ has been calculated in Section~\ref{sec:local_densities}, up to finitely many $p$-local densities $\delta_p(Q,-k)$.
However, all these terms are computed in \cite{Yang}.
In the complex case we have to use Lemma~\ref{lem3:delta_H_C}.
We usually use the \emph{advanced quadratic forms library} in Sage~\cite{Sage}, programmed by J.~Hanke and A.~Haensch.

For example, for $Q=I_{4,1}$ and $k=1$ considered in Example~\ref{ex1:Q=I_41} we have that $\coeff[I_{4,1}]=2\pi^2/3$ by \eqref{eq4:C'(Q)}  and, by Theorem~\ref{thm3:real-density},
\[
\delta(I_{4,1},-1)=\frac{L(2,4)}{\zeta(4)} \; \frac{\delta_2(I_{4,1},-1)}{1-2^{-4}} = \frac{\pi^2/8}{\pi^4/90}\;\frac{5/8}{15/2^{4}}=\frac{15}{2}\pi^{-2}.
\]
Hence
$
\coef[I_{4,1},-1]=5.
$
In Example~\ref{ex1:H=I_51} we take $H=I_{5,1}$ over $\Z[\frac{1+\sqrt{-3}}{2}]$ and $k=1$, then we have that $\coeff=(2\pi)^6/(5!\,\,3^3)$ by \eqref{eq4:C'(Q)} and $\delta(I_{5,1},-1) = 1/\zeta(6)\,.\, \delta_2(I_{5,1},-1)/(1-2^{-6})\,.\, \delta_3(I_{5,1},-1)/(1-3^{-6})$ by \eqref{eq3:delta(H,k)_m_odd}.
We check at once that $\coef[I_{5,1},-1]=18$.

\subsection{Calculation of $N_t(Q,-k)$}
We consider the functions $F_{A,\mathcal O}(m) := \#\{x\in\mathcal O^n:A[x]=l\} $ and $G_{\mathcal O}(m) := \#\{z\in\mathcal O: |z|^2=m\}$.
It follows by definition \eqref{eq1:N_t} that
\begin{equation}\label{eq4:N_t=sumF_A}
N_t(Q,-k)=\sum_{m=1}^{[t^2]}  F_A(am-k)\; G_{\mathcal O}(m).
\end{equation}
In the real case, \eqref{eq4:N_t=sumF_A} reduces as $N_t(Q,-k)=2\,\sum_{m=1}^{[t]}  F_A(am^2-k)$.

There exist formulas for $F_{A,\mathcal O}(m)$ and $G_{\mathcal O}(m)$ for some choices of $A$ and $\mathcal O$, that allow us to have fast algorithms to compute $N_t(Q,-k)$ for large $t$.
For example, we have used Jacobi's four-square formula
\[
F_{I_4,\Z}(m)= \sum_{d\mid m,\, 4\nmid d} \, d
\]
in Example~\ref{ex1:Q=I_41}.
The other formulas for $F_A(m)$ that we use here can find in \cite{Williams08}, \cite{Williams10} and \cite{Otremba}.

\subsection{Estimation of the decay order of $\Psi_t(Q,-k)$} \label{subsec:est_sigma}
We have computed in the last two subsections the necessary elements to obtain the function $\Psi(t)$ given in \eqref{eq4:Psi(t)}.
This function goes to zero when $t\to+\infty$.
Our purpose is to estimate the order of decay of this function.
The method used here is new as far as we know.

We consider the set of points in the plane given by
\[
\mathcal E=\left\{\big(m,\Psi(m)\big) \in [\![1,T]\!]\times\R: \Psi(m)> \Psi(l) \;\;\forall\, l>m\right\}.
\]
Now, by the ``least square fitting--power law'' method, we approximate these points by a curve $y(t)=B t^{\sigma}$ with $B>0$ and $\sigma<0$.
This method works as follows:
if $\{(x_i,y_i)\}_{1\leq i\leq h}$ are points with positive entries, then $B$ and $\sigma$ are given by
\begin{align*}
\sigma &= \frac
    {\displaystyle h\sum_{i=1}^h (\log x_i\log y_i) -\sum_{i=1}^h (\log x_i)\sum_{i=1}^h(\log y_i)}
    {\displaystyle h\sum_{i=1}^h (\log x_i)^2 - \Big(\sum_{i=1}^h \log x_i\Big)^2},\\
B &= \dfrac1h\left(\displaystyle \sum_{i=1}^h (\log y_i) - \sigma \sum_{i=1}^h (\log x_i)\right).
\end{align*}
Finally, our estimate of the decay order of $\Psi(t)$ is $O(t^\sigma)$.

For instance, in Example~\ref{ex1:Q=I_41} we obtain the curve $y(t)=13\,t^{-0.93}$ and in Example~\ref{ex1:H=I_51} the curve $y(t)=56\, t^{-1.36}$.

\section{Results and conclusions}\label{sec:conclusion}
We conclude the paper by showing the experimental results in several tables, together with some conclusions.
Recall that formula \eqref{eq1:main-form} proved in \cite{Lauret12} implies that $\Psi(t)=\Psi_t(Q,-k)\ll t^{\tau-2\rho}$ where $\tau$ is defined by \eqref{eq1:tau}.
Hence $\lim_{t\to\infty}\Psi(t)=0$ since $\tau<2\rho$.

In Corollary~\ref{cor2:cota_tau_R} and Remark~\ref{rem2:cota_tau_C} we obtain that
\[
\tau-2\rho\leq \Upsilon:=
\begin{cases}
  -1/3  &\quad\text{if $\F=\R$ and $n=2$},\\
  -1/2  &\quad\text{if $\F=\R$ and $n\geq3$},\\
  -1    &\quad\text{if $\F=\C$ and $n\geq2$}.
\end{cases}
\]
Recall that, in the complex case, we are assuming the extended result of Theorem~\ref{thm2:lambda_1_C}.
Similarly, from \eqref{eq1:tau_sin_autov_excep}, by assuming that condition \eqref{eq1:weaker_cond} holds (e.g.\ there are no exceptional eigenvalues) then
\begin{align*}
\tau-2\rho=\Omega :=&\,
\begin{cases}
  -1+\dfrac{2}{n+1}
    &\quad\text{if $\F=\R$},\\[4mm]
  -2+\dfrac{2}{n+1}+\varepsilon
    &\quad\text{if $\F=\C$},
\end{cases}
\end{align*}
for each $\varepsilon>0$.
On the other hand, we have estimated the decay order of $\Psi_t(Q,-k)$ by a curve $y(t)=B\, t^\sigma$ (see Subsection~\ref{subsec:est_sigma}).
Now, our purpose is to compare the number $\sigma$ with $\Upsilon$ and $\Omega$.

In Tables~\ref{table:F=R_A=(1,1)_k=1}, \ref{table:F=R_A=(1,1,1,1)_k=1}, \ref{table:F=R_A=(1,1,1,1,1,1)_k=1} and \ref{table:F=R_A=(1,1,1,1,1,1,1,1)_k=1} we consider the $\R$-hermitian forms (or quadratic forms)
\[
Q=
\begin{pmatrix}
I_n\\&-a
\end{pmatrix}\qquad \text{with }\quad1\leq a\leq 15,
\]
for $n=2,4,6,8$ respectively ($I_n$ denotes the $n\times n$-identity matrix).
Similarly, Tables~\ref{table:F=C_d=-3_A=(1,1)_k=1}, \ref{table:F=C_d=-3_A=(1,1,1)_k=1}, \ref{table:F=C_d=-3_A=(1,1,1,1)_k=1}, \ref{table:F=C_d=-3_A=(1,1,1,1,1)_k=1} consider the $\C$-hermitian forms (or just hermitian forms) $H=\diag(I_n,-a)$ with $1\leq a\leq 15$, over the ring $\mathcal O=\Z[\sqrt{-3}]$ for $n=2,3,4,5$ respectively.
In each table is fixed $k=1$, and $T=10^6$ if $\F=\R$ and $T=10^4$ if $\F=\C$.

{
\setlength{\arraycolsep}{3mm}

We abbreviate the results from Tables~\ref{table:F=R_A=(1,1)_k=1}--\ref{table:F=C_d=-3_A=(1,1,1,1,1)_k=1} as follows:
\begin{equation}\label{eq5:resumen_tablas}
\begin{array}{ccccr@{}lc}
\F& n&\multicolumn{1}{c}{\Upsilon_{n,\F}}&\multicolumn{1}{c} {\Omega_{n,\F}}& \multicolumn{2}{c}{\sigma\text{ in}} &\text{Table}\\ \hline\hline
\rule{0pt}{2.5ex}
&2 &-0.33 & -0.33 & [-0.68&,-0.50]   & \ref{table:F=R_A=(1,1)_k=1}              \\[1mm]
&4 &-0.50 & -0.60 & [-0.98&,-0.93]   & \ref{table:F=R_A=(1,1,1,1)_k=1}          \\[1mm]
\raisebox{1.5ex}[0pt]{$\R$}
&6 &-0.50 & -0.71 & [-0.999&,-0.991] & \ref{table:F=R_A=(1,1,1,1,1,1)_k=1}      \\[1mm]
&8 &-0.50 & -0.78 & [-1.002&,-0.9976]& \ref{table:F=R_A=(1,1,1,1,1,1,1,1)_k=1}  \\[1mm] \hline
\rule{0pt}{2.5ex}
&2 &-1 & -1.33 & [-1.40&,-1.33] & \ref{table:F=C_d=-3_A=(1,1)_k=1}              \\[1mm]
&3 &-1 & -1.50 & [-1.37&,-1.32] & \ref{table:F=C_d=-3_A=(1,1,1)_k=1}            \\[1mm]
\raisebox{1.5ex}[0pt]{$\C$}
&4 &-1 & -1.60 & [-1.37&,-1.33] & \ref{table:F=C_d=-3_A=(1,1,1,1)_k=1}          \\[1mm]
&5 &-1 & -1.67 & [-1.37&,-1.34] & \ref{table:F=C_d=-3_A=(1,1,1,1,1)_k=1}        \\[1mm] \hline
\end{array}
\end{equation}
}

We divide the conclusions according to $\sigma$ is larger or not than $\Omega$, or equivalently, the numerical data gives evidences of exceptional spectrum or of a new lower bound for $\lambda_1$.

\subsection{Evidences of exceptional spectrum}
We see from Tables~\ref{table:F=C_d=-3_A=(1,1,1)_k=1}--\ref{table:F=C_d=-3_A=(1,1,1,1,1)_k=1} that $\sigma>\Omega_{n,\C}$ for any $n=3,4,5$, hence these approximations say that \eqref{eq1:weaker_cond} should not hold, in particular, there are exceptional eigenvalues in $\Gamma_Q^0\ba\mathrm H_\C^n$. This numerical data supports Conjecture~\ref{conj1:existence}.

Let $\upsilon=\upsilon_{n,\C}= 2n-1$ and $\omega=\omega_{n,\C} = 4n^3/(n+1)^2$.
Hence, $\lambda_1\geq\upsilon$ is the extended lower bound proved in Theorem~\ref{thm2:lambda_1_C} and $\lambda_1\geq \omega$ is the condition \eqref{eq1:weaker_cond}.
When $n=2$, since $\sigma<\Omega_{2,\C}$ then $\lambda_1\geq \omega$.
For $3\leq n\leq 5$, $\sigma$ lies (approximately) in $[-1.4,-1.33]$ by \eqref{eq5:resumen_tablas}.
Hence, these approximations tell us that $\tau-2\rho = \sqrt{\rho^2-\lambda_1}-\rho$ lies in the same interval, or equivalently, $\lambda_1$ lies in the intervals showed as follows:
\begin{equation}\label{eq5:values_spectrum_complejo}
\begin{array}{c@{\qquad}c@{\qquad}c@{\qquad}r@{\;}l@{\qquad}c}
n & \upsilon & \omega & \multicolumn{2}{c}{\lambda_1\;\text{ in}\qquad} & \rho^2\\ \hline
2 & 3 &  3.55 & [ 3.55,& 4]  &  4 \\
3 & 5 &  6.75 & [ 6.22,& 6.44]  &  9 \\
4 & 7 & 10.24 & [ 8.88,& 9.24]  & 16 \\
5 & 9 & 13.88 & [11.55,&12.04]  & 25 \\
\end{array}
\end{equation}
\setlength{\unitlength}{9mm}
\newcommand{\raya}{\circle*{0.1}}
\begin{center}
\begin{picture}(10,4.05)
\multiput(0,0.5)(0,1){4}{\line(1,0){10}}
\multiput(0,0.5)(0,1){4}{\raya}
\multiput(-0.1,0)(0,1){4}{$0$}
\multiput(10,0.5)(0,1){4}{\raya}
\put(9.9,3){$\rho^2$}
\put(9.9,2){$\rho^2$}
\put(9.9,1){$\rho^2$}
\put(9.9,0){$\rho^2$}
\put(-1.5,0.4){$n=5$}
\put(-1.5,1.4){$n=4$}
\put(-1.5,2.4){$n=3$}
\put(-1.5,3.4){$n=2$}
\put(8.88,3.5){\raya}   \put(8.78,3){$\omega$}
\put(7.50,2.5){\raya}   \put(7.40,2){$\omega$}
\put(6.40,1.5){\raya}   \put(6.30,1){$\omega$}
\put(5.55,0.5){\raya}   \put(5.45,0){$\omega$}
\put(7.500,3.5){\raya}  \put(7.400,3){$\upsilon$}
\put(5.550,2.5){\raya}  \put(5.450,2){$\upsilon$}
\put(4.375,1.5){\raya}  \put(4.275,1){$\upsilon$}
\put(3.600,0.5){\raya}  \put(3.500,0){$\upsilon$}
\put(9.44  ,3.5){\oval(1.12  ,0.2)} \put(9.34  ,3){$\lambda_1$}
\put(7.0345,2.5){\oval(0.2419,0.2)} \put(6.93  ,2){$\lambda_1$}
\put(5.6625,1.5){\oval(0.2194,0.2)} \put(5.5625,1){$\lambda_1$}
\put(4.7191,0.5){\oval(0.1937,0.2)} \put(4.618 ,0){$\lambda_1$}
\end{picture}
\end{center}

\begin{table}
\begin{minipage}[t]{0.48\textwidth}
\caption{}\label{table:F=R_A=(1,1)_k=1}
$Q=\diag(I_2,-a)\quad  1\leq a\leq 15$,\\
$\F=\R,\; n=2,\; T=10^6,$\\[1mm]
$\Upsilon_{2,\R}=-0.33, \qquad \Omega_{2,\R}=-0.33.$\\[4mm]
$
\begin{array}{cccc}
a&C(Q,-1)&\Psi(T)&\sigma \\ \hline
1 & 2                       & 0.000492 & -0.51 \\
2 & 4\sqrt{2}               & 0.001647 & -0.56 \\
3 & 4\sqrt{3}               & 0.000691 & -0.65 \\
4 & 0\\
5 & \frac{4}{3}\sqrt{5}     & 0.000427 & -0.68 \\
6 & 4\sqrt{6}               & 0.002946 & -0.60 \\
7 & 0\\
8 & 0\\
9 & 4                       & 0.001732 & -0.60 \\
10 & \frac{4}{3}\sqrt{10}   & 0.001229 & -0.57 \\
11 & \frac{12}{5}\sqrt{11}  & 0.002404 & -0.57 \\
12 & 0\\
13 & \frac{4}{7}\sqrt{13}   & 0.001848 & -0.50 \\
14 & \frac{8}{3}\sqrt{14}   & 0.003855 & -0.50 \\
15 & 0
\end{array}
$
\end{minipage}
\begin{minipage}[t]{0.48\textwidth}
\caption{}\label{table:F=R_A=(1,1,1,1)_k=1}
$Q=\diag(I_4,-a)\quad  1\leq a\leq 15$,\\
$\F=\R,\; n=4,\; T=10^6,$\\[1mm]
$\Upsilon_{4,\R}=-0.5, \qquad \Omega_{4,\R}=-0.6.$\\[4mm]
$
\begin{array}{cccc}
a&C(Q,-1)&\Psi(T)&\sigma \\ \hline
1 & 5                           & 0.000004 & -0.93 \\
2 & 8\sqrt{2}                   & 0.000015 & -0.97 \\
3 & 12\sqrt{3}                  & 0.000032 & -0.98 \\
4 & 32                          & 0.000039 & -0.95 \\
5 & \frac{140}{13}\sqrt{5}      & 0.000030 & -0.98 \\
6 & \frac{72}{5}\sqrt{6}        & 0.000059 & -0.98 \\
7 & \frac{112}{5}\sqrt{7}       & 0.000082 & -0.98 \\
8 & 32\sqrt{2}                  & 0.000060 & -0.97 \\
9 & 36                          & 0.000008 & -0.97 \\
10 & \frac{280}{13}\sqrt{10}    & 0.000072 & -0.98 \\
11 & \frac{1540}{61}\sqrt{11}   & 0.000164 & -0.97 \\
12 & \frac{192}{5}\sqrt{3}      & 0.000101 & -0.97 \\
13 & \frac{364}{17}\sqrt{13}    & 0.000073 & -0.97 \\
14 & \frac{112}{5}\sqrt{14}     & 0.000136 & -0.97 \\
15 & \frac{360}{13}\sqrt{15}    & 0.000200 & -0.98
\end{array}
$
\end{minipage}
\\[10mm]
\begin{minipage}[t]{0.48\textwidth}
\caption{}\label{table:F=R_A=(1,1,1,1,1,1)_k=1}
$Q=\diag(I_6,-a)\quad  1\leq a\leq 15$,\\
$\F=\R,\; n=6,\; T=10^6,$\\[1mm]
$\Upsilon_{6,\R}=-0.50, \qquad \Omega_{6,\R}=-0.71.$\\[4mm]
$
\begin{array}{cccc}
a&C(Q,-1)&\Psi(T)&\sigma \\ \hline
1 & \frac{27}{4}                    & 0.000016 & -0.992 \\\
2 & 18\sqrt{2}                      & 0.000048 & -0.998 \\\
3 & \frac{945}{26}\sqrt{3}          & 0.000135 & -0.998 \\\
4 & 120                             & 0.000327 & -0.991 \\\
5 & \frac{1125}{14}\sqrt{5}         & 0.000400 & -0.998 \\\
6 & \frac{1242}{13}\sqrt{6}         & 0.000450 & -0.999 \\\
7 & \frac{7252}{57}\sqrt{7}         & 0.000832 & -0.998 \\\
8 & 360\sqrt{2}                     & 0.001283 & -0.999 \\
9 & \frac{15309}{26}                & 0.001346 & -0.996 \\
10 & \frac{3950}{21}\sqrt{10}       & 0.001198 & -0.998 \\
11 & \frac{9801}{38}\sqrt{11}       & 0.001887 & -0.997 \\
12 & \frac{8640}{13}\sqrt{3}        & 0.002874 & -0.998 \\
13 & \frac{689013}{2198}\sqrt{13}   & 0.002703 & -0.997 \\
14 & \frac{6468}{19}\sqrt{14}       & 0.002454 & -0.998 \\
15 & \frac{5550}{13}\sqrt{15}       & 0.003827 & -0.998 \\
\end{array}
$
\end{minipage}
\begin{minipage}[t]{0.48\textwidth}
\caption{}\label{table:F=R_A=(1,1,1,1,1,1,1,1)_k=1}
$Q=\diag(I_8,-a)\quad  1\leq a\leq 15$,\\
$\F=\R,\; n=8,\; T=10^6,$\\[1mm]
$\Upsilon_{8,\R}=-0.5, \qquad \Omega_{8,\R}=-0.78.$\\[4mm]
$
\begin{array}{cccc}
a&C(Q,-1)&\Psi(T)&\sigma\\ \hline
1 & \frac{685}{136}                     & 0.000018 & -0.9976  \\
2 & \frac{440}{17}\sqrt{2}              & 0.000127 & -0.9999  \\
3 & \frac{46575}{697}\sqrt{3}           & 0.000377 & -1.0000  \\
4 & \frac{5120}{17}                     & 0.001031 & -0.9991  \\
5 & \frac{84375}{313}\sqrt{5}           & 0.002218 & -1.0001  \\
6 & \frac{281880}{697}\sqrt{6}          & 0.003478 & -1.0001  \\
7 & \frac{11627700}{20417}\sqrt{7}      & 0.004801 & -1.0001  \\
8 & \frac{28160}{17}\sqrt{2}            & 0.008191 & -1.0002  \\
9 & \frac{2496825}{697}                 & 0.013384 & -0.9996  \\
10 & \frac{7885000}{5321}\sqrt{10}      & 0.015999 & -1.0001  \\
11 & \frac{214923225}{124457}\sqrt{11}  & 0.018761 & -0.9999  \\
12 & \frac{3179520}{697}\sqrt{3}        & 0.027650 & -1.0001  \\
13 & \frac{43006275}{14281}\sqrt{13}    & 0.039000 & -1.0001  \\
14 & \frac{68682320}{20417}\sqrt{14}    & 0.044185 & -1.0002  \\
15 & \frac{815568750}{218161}\sqrt{15}  & 0.047254 & -1.0001  \\
\end{array}
$
\end{minipage}
\end{table}

\begin{table}
\begin{minipage}[t]{0.48\textwidth}
\caption{}\label{table:F=C_d=-3_A=(1,1)_k=1}
$Q=\diag(I_2,-a)\quad  1\leq a\leq 15$,\\
$\F=\C,\; \mathcal O=\Z[\sqrt{-3}],\; n=2,\; T=10^4,$\\[1mm]
$\Upsilon_{2,\C}=-1, \qquad \Omega_{2,\C}=-1.33.$\\[4mm]
$
\begin{array}{cccc}
a&C(Q,-1)&\Psi(T)&\sigma\\ \hline
1 & 18                  & 0.000039 & -1.33  \\
2 & 48                  & 0.000061 & -1.34  \\
3 & 108                 & 0.000102 & -1.36  \\
4 & 48                  & 0.000147 & -1.37  \\
5 & \frac{1200}{7}      & 0.000050 & -1.35  \\
6 & 144                 & 0.000255 & -1.34  \\
7 & \frac{2352}{19}     & 0.000196 & -1.40  \\
8 & 192                 & 0.000243 & -1.35  \\
9 & 324                 & 0.000100 & -1.37  \\
10 & \frac{800}{7}      & 0.000296 & -1.35  \\
11 & \frac{14520}{37}   & 0.000686 & -1.36  \\
12 & 288                & 0.000345 & -1.34  \\
13 & \frac{14196}{61}   & 0.000571 & -1.36  \\
14 & \frac{6272}{19}    & 0.000388 & -1.37  \\
15 & \frac{3600}{7}     & 0.000407 & -1.37  \\
\end{array}
$
\end{minipage}
\begin{minipage}[t]{0.48\textwidth}
\caption{}\label{table:F=C_d=-3_A=(1,1,1)_k=1}
$Q=\diag(I_3,-a)\quad  1\leq a\leq 15$,\\
$\F=\C,\; \mathcal O=\Z[\sqrt{-3}],\; n=3,\; T=10^4,$\\[1mm]
$\Upsilon_{3,\C}=-1, \qquad \Omega_{3,\C}=-1.5.$\\[4mm]
$
\begin{array}{cccc}
a&C(Q,-1)&\Psi(T)&\sigma\\ \hline
1 & 30                  & 0.000085 & -1.35  \\
2 & \frac{576}{5}       & 0.000349 & -1.32  \\
3 & 324                 & 0.000884 & -1.35  \\
4 & 576                 & 0.001443 & -1.35  \\
5 & \frac{7875}{13}     & 0.002199 & -1.36  \\
6 & \frac{7776}{5}      & 0.003006 & -1.36  \\
7 & \frac{58653}{40}    & 0.004731 & -1.36  \\
8 & \frac{9216}{5}      & 0.005813 & -1.34  \\
9 & 2916                & 0.007188 & -1.36  \\
10 & \frac{47250}{13}   & 0.008708 & -1.35  \\
11 & \frac{886446}{305} & 0.010820 & -1.36  \\
12 & \frac{31104}{5}    & 0.011752 & -1.37  \\
13 & \frac{1206153}{238}& 0.014859 & -1.35  \\
14 & \frac{703836}{125} & 0.017702 & -1.34  \\
15 & \frac{212625}{26}  & 0.022768 & -1.36  \\
\end{array}
$
\end{minipage}
\\[10mm]
\begin{minipage}[t]{0.48\textwidth}
\caption{}\label{table:F=C_d=-3_A=(1,1,1,1)_k=1}
$Q=\diag(I_4,-a)\quad  1\leq a\leq 15$,\\
$\F=\C,\; \mathcal O=\Z[\sqrt{-3}],\; n=4,\; T=10^4,$\\[1mm]
$\Upsilon_{4,\C}=-1, \qquad \Omega_{4,\C}=-1.6.$\\[4mm]
$
\begin{array}{cccc}
a&C(Q,-1)&\Psi(T)&\sigma\\ \hline
1 & 30                          & 0.000052 & -1.36 \\
2 & \frac{1920}{11}             & 0.000468 & -1.33 \\
3 & 648                         & 0.001536 & -1.35 \\
4 & \frac{19200}{11}            & 0.003596 & -1.35 \\
5 & \frac{1560000}{521}         & 0.006766 & -1.35 \\
6 & \frac{51840}{11}            & 0.012691 & -1.36 \\
7 & \frac{28812000}{2801}       & 0.018688 & -1.36 \\
8 & \frac{122880}{11}           & 0.029712 & -1.36 \\
9 & 17496                       & 0.039107 & -1.36 \\
10 & \frac{156000000}{5731}     & 0.057215 & -1.36 \\
11 & \frac{428688480}{13421}    & 0.076104 & -1.36 \\
12 & \frac{414720}{11}          & 0.100624 & -1.36 \\
13 & \frac{2039255400}{30941}   & 0.111032 & -1.37 \\
14 & \frac{1843968000}{30811}   & 0.158689 & -1.37 \\
15 & \frac{42120000}{521}       & 0.189703 & -1.36 \\
\end{array}
$
\end{minipage}
\begin{minipage}[t]{0.48\textwidth}
\caption{}\label{table:F=C_d=-3_A=(1,1,1,1,1)_k=1}
$Q=\diag(I_5,-a)\quad  1\leq a\leq 15$,\\
$\F=\C,\; \mathcal O=\Z[\sqrt{-1}],\; n=5,\; T=10^4,$\\[1mm]
$\Upsilon_{5,\C}=-1, \qquad \Omega_{5,\C}=-1.67.$\\[4mm]
$
\begin{array}{cccc}
a&C(Q,-1)&\Psi(T)&\sigma\\ \hline
1 & 18                          & 0.000059 & -1.36  \\
2 & \frac{4224}{13}             & 0.000920 & -1.34  \\
3 & \frac{17496}{13}            & 0.004840 & -1.35  \\
4 & \frac{33792}{7}             & 0.014810 & -1.35  \\
5 & \frac{4884375}{403}         & 0.036694 & -1.36  \\
6 & \frac{2052864}{91}          & 0.075213 & -1.36  \\
7 & \frac{141229221}{3268}      & 0.142486 & -1.37  \\
8 & \frac{1081344}{13}          & 0.234356 & -1.36  \\
9 & \frac{1417176}{13}          & 0.387823 & -1.36  \\
10 & \frac{286550000}{1519}     & 0.577332 & -1.37  \\
11 & \frac{12968792826}{45695}  & 0.866876 & -1.37  \\
12 & \frac{32845824}{91}        & 1.201558 & -1.36  \\
13 & \frac{34464530139}{67039}  & 1.680102 & -1.37  \\
14 & \frac{8285447632}{10621}   & 2.199171 & -1.37  \\
15 & \frac{2373806250}{2821}    & 3.014865 & -1.36  \\
\end{array}
$
\end{minipage}
\end{table}

\subsection{Evidences on a lower bound for $\lambda_1$}
We now restrict our attention when $\sigma$ is smaller than $\Omega$ and $\Upsilon$, namely, the cases $\F=\R$ with $n=2,4,6,8$ and $\F=\C$ with $n=2$.
As we explained in Introduction, these data tell us that the error term of \eqref{eq1:main-form} is better than the error term expected by the lattice point theorem, even by assuming the nonexistence of exceptional eigenvalues in $\Gamma_Q^0\ba\mathrm H_\R^n$.
Therefore, these computations are numerical evidence that supports Conjecture~\ref{conj1:lower_bound} in these cases.

Some ``noise'' is visible in the computation of $n=2$.
Actually, the amplitude of the values taken by $\sigma$ is dramatically smaller as $n$ grows.
This noise makes us doubt about how good is the estimation $O(t^\sigma)$ of the error term in \eqref{eq1:main-form} when $n=2$.
However, we do not conjecture anything in this case ($n=2$) since the lower bound obtained in \cite{Kim-Sarnak} gives the best error term in \eqref{eq1:main-form} that we can obtain form \cite{Lauret12}.
In other words, when $\F=\R$ and $n=2$, we are showing evidences for the lower bound \eqref{eq1:weaker_cond}, namely $\lambda_1\geq 2/9=0.22\cdots$, which was already proved by Kim and Sarnak~\cite{Kim-Sarnak}, namely $\lambda_1\geq 975/4096=0.238\cdots$.
Perhaps, the complications that emerge in this case ($n=2$) are connected with the same complexities of the \emph{Gauss circle problem}.

We actually consider some other $\F$-hermitian forms that are not included here.
They are $\diag(1,1,2,2,-a)$, $\diag(1,1,1,4,-a)$, $\diag(1,1,3,3,-a)$, $\diag(1,2,2,4,-a)$, $\diag(1,4,4,4,-a)$, $\diag(1,1,4,4,-a)$, $\diag(1,1,1,1,2,2,-a)$ and $\diag(1,1,2,2,2,2,-a)$ when $\F=\R$ and also $\diag(I_n,-a)$ over $\Z[\sqrt{-1}]$ for $n=2,3,4$ when $\F=\C$, with $1\leq a \leq 15$.
We have used formulas for $F_A(m)$ from \cite{Williams07} and \cite{Williams08}.
The numerical results in these cases are very similar those exposed here.

\section*{Acknowledgments}
The author sincerely thanks Roberto Miatello for many helpful discussions concerning the material in this paper.

\bibliographystyle{plain}

\end{document}